\documentclass{article}
\usepackage{amsmath}
\usepackage{amsfonts}
\usepackage{amsthm}

\def\congruent{\equiv}

\def\ratQ{ \mathbb{Q} }  %used a kludge in plain TeX but here can use amsfont.
\def\realR{ \mathbb{R} } 
\def\notdiv{\not\vert\,}

\newtheorem{Theorem}{Theorem}
\newtheorem{Lemma}{Lemma}
\newtheorem{CorollaryThm1}{Corollary to Theorem }

\begin{document}

\begin{center}
{\large\bf 
The generalized Pillai equation $\pm r a^x \pm s b^y = c$ 
}
%third paper
\bigskip

Reese Scott

Robert Styer (correspondence author), Dept. of Mathematical Sciences, Villanova University, 800 Lancaster Avenue, Villanova, PA  19085--1699, phone 610--519--4845, fax 610--519--6928, robert.styer@villanova.edu
\end{center}

% revised 25 Feb 2010   19 Sept 2010   19 Jan 2011    
rev. 13 Feb 2011

\bigskip

\begin{abstract}  
In this paper we consider $N$, the number of solutions $(x,y,u,v)$ 
to the equation  $ (-1)^u r a^x + (-1)^v s b^y = c$
in nonnegative integers $x, y$ and integers $u, v \in \{0,1\}$,
for given integers $a>1$, $b>1$, $c>0$, $r>0$ and $s>0$. 
We show that $N \le 2$ when $\gcd(ra, sb) =1$ and $\min(x,y)>0$, except for a finite number of cases that can be found in a finite number of steps.  For arbitrary $\gcd(ra, sb)$ and $\min(x,y) \ge 0$, we show that when $(u,v) = (0,1)$ we have $N \le 3$, with an infinite number of cases for which $N=3$.   
\end{abstract}   

\bigskip

\section{Introduction}
 
In this paper we consider $N$, the number of solutions $(x,y,u,v)$ 
to the equation  
$$ (-1)^u r a^x + (-1)^v s b^y = c \eqno{(P)}$$
in nonnegative integers $x, y$ and integers $u, v \in \{0,1\}$,
for given integers $a>1$, $b>1$, $c>0$, $r>0$ and $s>0$.  

Equation (P) is a generalized form of the familiar Pillai equation $ r a^x - s b^y = c$, also referred to as Pillai's Diophantine equation as in \cite{BL} and \cite{W}.  This equation has been considered by many authors from a variety of standpoints; several subsets of the set $\{a,b,r,s,x,y\}$ have been treated as the unknowns.  Here we will consider only the case in which $a$, $b$, $r$, $s$, and $c$ are given.  For histories of other cases, as well as a more detailed history of the case under consideration here, we refer the reader to the surveys \cite{BBM} and \cite{W}.  

The result that $N$ is finite follows easily from an effective result of Ellison~\cite{E}, which is a response to an earlier result of Pillai~\cite{P1}.  Recent work has focused on finding small upper bounds on $N$.  

The case $rs=1$ has received particular attention.  Le~\cite{Le} showed that for $rs=1$, $N \le 2$ under the restrictions $x>1$, $y>1$, $a\ge 10^5$, $b \ge 10^5$, $u=0$, $v=1$, and $(a, b)=1$.  Bennett~\cite{Be} relaxed these restrictions to $x>0$, $y>0$, $a>1$, $b>1$, and later the authors~\cite{Sc-St2} removed the restrictions on $u$ and $v$, showing that $N \le 2$ whenever $rs=1$ and $\min(x,y)>0$, except for listed exceptional cases.  The results of \cite{Be} and \cite{Sc-St2} also remove the restriction $(a, b)=1$.  

The case $rs>1$ was treated by Shorey~\cite{Sh}, who showed that Equation (P) has at most 9 solutions in positive integers $(x,y)$ when $(u,v)=(0,1)$ and the terms on the left side of (P) are large relative to $c$.  Later, under the restrictions $x>1$, $y>1$, and $(ra, sb)=1$, Le~\cite{Le} obtained the following improved value for $N$:  

{\bf Theorem L \cite{Le}} 
If $u=0$, $v=1$, $x>1$, $y>1$, $a \ge e^e$, $b \ge e^e$, and $(ra, sb) =1$, then $N \le 3$.

In this paper we remove the restrictions $x>1$, $y>1$, $a \ge e^e$, $b \ge e^e$, and $(ra, sb)=1$ (see Theorem 3 in Section 3).  When $(ra, sb)=1$ and $\min(x,y)>0$, we show that $N \le 2$, even when $u$ and $v$ are unrestricted, except for a finite number of cases that can be found in a finite number of steps (see Theorem 1 in Section 2.)  The second author searched the ranges $1< b < a \le 100$, $1 \le r,s \le 1000$, $(ra,sb)=1$, $a \notdiv r$, $b \notdiv s$, $a$ and $b$ not perfect powers, looking for two solutions $(x,y)$ satisfying $1 \le x_1, x_2 \le 12$ and $1 \le y_1, y_2 \le 12$, and then looking for a third solution satisfying $x_3 \le 24$ and $y_3 \le 24$; the search found no cases of $N \ge 3$ other than $(a,b,c,r,s) = (3,2,1,1,1)$, $(3,2,5,1,1)$, $(3,2,7,1,1)$, $(3,2,11,1,1)$, $(3,2,13,1,1)$, and $(5,2,3,1,1)$.   We note that these exceptional $(a,b,c,r,s)$ all have $r=s=1$ and $a$ and $b$ both primes, so that they correspond to the list of exceptional cases in the elementary Theorem 7 of \cite{Sc-St1}.  

Theorem 1 uses a result of Matveev~\cite{Ma}, and Theorem 3 uses a recent result of He and Togb{\'e}~\cite{HT}.  
 
\section{$\gcd(ra,sb)=1$ and $\min(x,y)>0$, with $u$ and $v$ unrestricted}  %2

\begin{Theorem}  %1
Let $a>1$, $b>1$, $c>0$, $r>0$ and $s>0$ be positive integers with $(ra, sb)=1$.  
Then the equation
$$ (-1)^u r a^x + (-1)^v s b^y = c \eqno{(1)}$$
has at most two solutions $(x,y,u,v)$ in positive integers $x$, $y$ and integers $u, v \in \{0,1\}$, except for a finite number of cases which can be found in a finite number of steps.  More specifically, if (1) has more than two solutions, then $\max(a,b,r,s,x,y)< 8 \cdot 10^{14}$ for each solution.
\end{Theorem}

\begin{CorollaryThm1} 
Let $1 < b < a \le 15$, $c>0$, $1 \le r, s \le 100$ be positive integers with $(ra, sb)=1$.    
Then (1) has at most two solutions $(x,y,u,v)$ in positive integers $x$, $y$ and integers $u, v \in \{0,1\}$, unless $(a,b,c,r,s) = (3,2,1,1,1)$, $(3,2,5,1,1)$, $(3,2,5,1,2)$, $(3,2,7,1,1)$, $(3,2,11,1,1)$, $(3,2,13,1,1)$, $(3,2,13,1,2)$, $(4,3,13,1,1)$, or $(5,2,3,1,1)$.  
\end{CorollaryThm1}

Unlike the computer search mentioned in the introduction, the Corollary to Theorem 1 does not use the restrictions $a \notdiv r$, $b \notdiv s$, $a$ and $b$ not perfect powers. 

Throughout this section, we take $(ra,sb)=1$ and $\min(x,y)>0$.  

The proofs of Theorem 1 and its Corollary will follow several lemmas needed to prepare for the proof of Theorem 1.    

Since for each choice of $(x, y)$ giving a solution to (1), $(u,v)$ is determined, we will usually refer simply to a solution $(x,y)$. 

\begin{Lemma} %1
Let $a>1$, $b>1$, $r>0$, and $m>0$ be integers.  
Suppose there exist positive integers $y$ such that $(b^y \pm 1)/(r a^m)$ is an integer prime to $a$, and let $n$ be the least such $y$ possible regardless of the choice of sign.    
Then if $M > m$ and 
$$b^N \pm 1 = r a^M L $$
where the $\pm$ is independent of the above and $L$ and $a$ are not necessarily relatively prime, we must have 
$$ n {a^{M-m} \over 2^{g + h - 1}} \Bigm| N,$$
where $g=1$ and $h=0$ unless $r$ is odd, $a \congruent 2 \bmod 4$, and $m = 1$, in which case $g$ is the largest integer such that $2^g | b \pm 1$ (where the $\pm$ is chosen to maximize $g$), and $h$ is the largest integer such that $2^h | n$.  
\end{Lemma}

\begin{proof}
This lemma is essentially Lemma 1 of \cite{Sc-St2} with a slight generalization, and can be proven in essentially the same way.  
\end{proof}

\begin{Lemma}  % 2
If (1) has two solutions, then, if the value of $x$ is the same in both solutions, there is no third solution except when $(a,b,c,r,s) = (3,2,1,1,1)$, $(3,2,5,1,1)$, $(5,2,3,1,1)$, or $(3,2,7,1,1)$.  
\end{Lemma}

\begin{proof}  
Suppose (1) has two solutions, $(x_1, y_1)$ and $(x_2, y_2)$ with $x_1 = x_2$.  We can take $y_1 < y_2$.  Then 
$$2 r a^{x_1} = s b^{y_1} (b^h \pm 1) \eqno{(2)}$$
where $h = y_2 - y_1 > 0$.  We see that $b=2$, $s=y_1 = 1$, $r a^{x_1} = 2^h \pm 1$ and $c = 2^h \mp 1$.  If $h \le 2$, then $(a,b,c,r,s)$ must be one of the first three exceptional cases in the formulation of the lemma.  So we take $h>2$.  Now suppose (1) has a third solution $(x_3, y_3)$.  If $x_3=x_1$, then $(x_3, y_3)$ must be a duplicate solution.  If $x_3 < x_1$ then we have 
$$r a^{x_3} (a^{x_1 - x_3} \pm 1) = 2 ( 2^{y_3 - 1} \pm 1) \eqno{(3)}$$
where the $\pm$ signs are independent of each other and also independent of the sign in (2).  From (3) it follows that there exists a least number $n$ such that 
$$ 2^n \pm 1 = r a^{x_3} l \eqno{(4)}$$
where $(a,l) = 1$ and the sign in (4) is independent of the signs in (2) and (3).  By Lemma 1, there must be a prime $p | a$ such that $np | h$.  Also, $n \mid h$ implies $ra^{x_3} l \mid r a ^{x_1}$ (even when the sign in (4) does not match the sign in (2)) so that $l = 1$.  Now
$$2^{np} \pm 1 = r a^{x_3} p l_1 \eqno{(5)}$$
where $(l_1, r a)=1$ and the $\pm$ agrees with (4).  If $r a^{x_3} \ne 3$, we must have $l_1 > 1$, impossible since $np \mid h$ implies $r a^{x_3} p l_1 \mid r a^{x_1}$ (even if the sign in (5) does not match the sign in (2)).  So we must have $r a^{x_3} = 3$, forcing $2^h \pm 1 = 9$, yielding the fourth exceptional case in the formulation of the lemma.  

So we can assume $x_3 > x_1$.  Now
$$r a^{x_3} = \pm 2^{y_3} \pm c \eqno{(6)}$$
where the $\pm$ signs are independent.  Suppose $y_3 \le h$.  Then 
$$r a^{x_3} \ge 3 r a^{x_1} \ge 3 (2^h - 1) > 2^h+ (2^h +1) \ge 2^{y_3} + c$$
contradicting (6).  So $y_3 > h$.  
Letting $a^{x_3 - x_1} = 2^{k} t \pm 1$ with $t$ odd and $k \ge 2$, and considering (6) modulo $2^{h+1}$, we see that $k > h$ so that $x_3 - x_1$ is even (note $a \le 2^{h} +1 < 2^{h+1} - 1$).  So $a^{x_3-x_1} \congruent 1 \bmod 8$, and $ra^{x_1} \congruent -c \bmod 8$.  Now considering (6) modulo 8, we see that we must have 
$$ r a^{x_3} + c = 2^{y_3}$$
since no other combination of signs is possible.  But since we also have $ra^{x_1} + c = 2^{y_2}$ and $x_3 - x_1$ is even, we have a contradiction to Theorem 1 of \cite{Sc} since $c = 2^h \mp 1 \ne 3$, 5, or 13.  
\end{proof}

\begin{Lemma}  % 3
If (1) has two solutions $(x_1, y_1)$ and $(x_2, y_2)$ with $x_1 \le x_2$, we must have $r a^{x_2} > c/2$.  If in addition there is a third solution $(x_3, y_3)$ with $x_3 \ge x_2$ we must have $ r a^{x_3} > c$.    
\end{Lemma}

\begin{proof}
Assume (1) has two solutions $(x_1, y_1)$ and $(x_2, y_2)$ with $x_1 \le x_2$.  
Since $(ra, sb)=1$, $c$ is prime to $rasb$.  If $r a^{x_2} < c/2$, then $s b^{y_2} > c/2$ by (1), and also $r a^{x_1} < c/2$ so that $s b^{y_1} > c/2$.  But we have 
$$r a^{x_1} (a^{x_2 - x_1} \pm 1) = s b^{\min(y_1, y_2)} (b^{|y_2 - y_1|} \pm 1) $$
where the $\pm$ signs are independent, so that
$$c/2 > r a^{x_2} \ge a^{x_2 - x_1} +1 \ge s b^{\min(y_1, y_2)} > c/2,$$
a contradiction.  If there is a further solution $(x_3, y_3)$ with $x_3 \ge x_2$, then, by Lemma 2, we can take $x_1 < x_2 < x_3$, so that 
$$r a^{x_3} \ge 2 r a^{x_2} > c,$$
unless $(a,b,c,r,s)$ is one of the exceptional cases of Lemma 2, in which case it suffices to observe that we must have $a^{x_3} \ge a^2 \ge 9 > c$.  
\end{proof}

\begin{Lemma}  % 4
If (1) has three solutions $(x_1,y_1)$, $(x_2, y_2)$, and $(x_3, y_3)$ then 
$$c < (ZJ)^2$$
where $J = \max(a,b)$ and $Z = \max(x_1, x_2, x_3, y_1, y_2, y_3)$.
\end{Lemma}

\begin{proof}
Since $(r a, sb)=1$, we can take $sb$ odd without loss of generality.  Assume (1) has three solutions $(x_1, y_1)$, $(x_2, y_2)$, and $(x_3, y_3)$.  By Lemma 2 (noting that, for the exceptional cases of Lemma 2, $(ZJ)^2 > 9 > c$), we can take $x_1 < x_2 < x_3$ and we can further assume no two of $y_1$, $y_2$, and $y_3$ are equal.  Write $c^{1/t} = s b^{y_0}$ where $y_0=\min(y_1, y_2, y_3)$ and $t$ is some real number.  Let $i, j \in  \{1,2,3\}$ with $i<j$.   Then 
$$r a^{x_i} (a^{x_j - x_i} \pm 1) = s b^{\min(y_i, y_j)} (b^{|y_j - y_i|} \pm 1) \eqno{(7)}$$
where the $\pm$ signs are independent.  Let $n$ be the least positive integer such that $(b^n \pm 1)/ r a^{x_1}$ is an integer prime to $a$, where the $\pm$ sign is independent of (7) (that such $n$ exists follows from (7)).  From (7) we can apply Lemma 1 (with $g$ and $h$ defined as in Lemma 1) to get 
$$Z > |y_3 - y_2| \ge n { a^{x_2 - x_1} \over 2^{g+h-1} } \ge {a^{x_2 - x_1} \over 2^{g-1}} \ge { s b^{y_0} - 1 \over (b+1)/2 } = {c^{1/t} -1 \over (b+1)/2}  \ge {c^{1/t} \over b}  $$
since $b \ge 3$.  Thus, 
$$c < (bZ)^t. \eqno{(8)}$$

Now reorder the three solutions $(x,y)$ as $(X_1, Y_1)$, $(X_2, Y_2)$, $(X_3, Y_3)$ so that $Y_1 < Y_2 < Y_3$. Now, again using Lemma 1, we have 
$$ Z > | X_3 - X_2 | \ge b^{Y_2-Y_1} = {s b^{Y_2} \over s b^{Y_1} } > { c/2 \over c^{1/t} } $$
where the last inequality follows from Lemma 3 (with the roles of $a$ and $b$ reversed).  So now  
$$ c < (2Z)^{{t\over t-1}}. \eqno{(9)}$$
Using (8) if $t \le 2$ and using (9) if $t > 2$, we see that the lemma holds.  
\end{proof}

\begin{Lemma} %5
If (1) has three solutions $(x_1,y_1)$, $(x_2, y_2)$, and $(x_3, y_3)$, then $\max(x_1$, $y_1$, $x_2$, $y_2$, $x_3$, $y_3)$ $\ge \max(r,s,a,b)$.  
\end{Lemma}

\begin{proof}
Assume (1) has three solutions $(x_1, y_1)$, $(x_2, y_2)$, and $(x_3, y_3)$.  By Lemma 2 (noting that for the exceptional cases of Lemma 2 we must have $y_3 \ge a = \max(r,s,a,b)$), we can take $x_1 < x_2 < x_3$ and can further assume no two of $y_1$, $y_2$, and $y_3$ are equal.  Let $n$ be the least positive integer such that $(b^n \pm 1)/ r a^{x_1}$ is an integer prime to $a$ (that such $n$ exists follows from (7)).  Let $Z = \max(x_1, y_1, x_2, y_2, x_3, y_3)$.  We can apply Lemma 1 and (7) to get 
$$Z > |y_3 - y_2| \ge n { a^{x_2 - x_1} \over 2^{g+h-1} } \ge {a^{x_2 - x_1} \over 2^{g-1}} \ge { s b - 1 \over 2^{g-1} },  $$
where $g$ and $h$ are as in Lemma 1.  If $a$ is odd, then $g=1$ and 
$$Z \ge \max(a,b,s).$$
If $a$ is even, then, using $2^{g-1} \le (b+1)/2$ and $b \ge 3$, we get 
$$Z > { s b - 1 \over (b+1)/2 } \ge s.$$
In the same manner, reversing the roles of $a$ and $b$, we get, when $b$ is odd, 
$$Z \ge \max(a,b, r),$$
and, when $b$ is even, we get
$$ Z > r.$$
Since $(ra, sb) = 1$, at least one of $a$ or $b$ must be odd, so that the lemma holds.  
\end{proof}

\begin{Lemma}  %6
{\rm (Matveev~\cite{Ma} as given in \cite[Theorem 1]{Mi})}
Let $\lambda_1$, $\lambda_2$, $\lambda_3$ be $\ratQ$-linearly independent logarithms of non-zero algebraic numbers and let $b_1$, $b_2$, $b_3$ be rational integers with $b_1 \ne 0$.  Define $\alpha_j = \exp(\lambda_j)$ for $j=1$, 2, 3 and
$$\Lambda  = b_1 \lambda_1 + b_2 \lambda_2 + b_3 \lambda_3.$$
Let $D$ be the degree of the number field $\ratQ(\alpha_1, \alpha_2, \alpha_3)$ over $\ratQ$.  Put
$$\chi = [\realR(\alpha_1, \alpha_2, \alpha_3):\realR].$$
Let $A_1$, $A_2$, $A_3$ be positive real numbers, which satisfy 
$$A_j \ge \max\{ D h(\alpha_j), |\lambda_j|, 0.16\} \qquad (1 \le j \le 3).$$
Assume that
$$B \ge \max\{ 1, \max\{ |b_j|A_j/A_1 : 1 \le j \le 3 \} \}.$$
Define also 
$$C_1 = { 5 \times 16^5 \over 6 \chi}  e^3 (7 + 2 \chi) \left({3e \over 2}\right)^\chi \Big(20.2 + \log\big(3^{5.5} D^2 \log(eD)\big)\Big).$$
Then
$$\log|\Lambda| > -C_1 D^2 A_1 A_2 A_3 \log\big(1.5 e D B \log(eD)\big).$$
\end{Lemma}

\begin{proof}[Proof of Theorem 1]
Recall that, for a given solution, $(x,y)$ uniquely determines $(u,v)$, so we refer to a solution $(x,y)$.  Assume (1) has three solutions $(x_1,y_1)$, $(x_2, y_2)$, and $(x_3, y_3)$.  
Let $Z = \max(x_1, y_1, x_2, y_2, x_3, y_3)$.  Without loss of generality, we can take $Z=\max(x_3, y_3)$.  Also let $D= \max(r a^{x_3}, s b^{y_3})$, $d = \min(r a^{x_3}, s b^{y_3})$, $J = \max(a,b)$, and $j = \min(a,b)$.  By Lemma 3, $D > c$, so that 
$$|r a^{x_3} - s b^{y_3}| = D - d = c.$$
Let 
$$\Lambda = |\log(r/s)+ x_3 \log(a)-y_3 \log(b)|=\log(D) - \log(d)= \log(1 + c/d)< c/d$$
so that 
$$\log(1/\Lambda) > \log(d) - \log(c).  \eqno{(10)}$$
Theorem 1 of \cite{Sc-St2}, along with the list of solutions in Theorem 6 of \cite{Sc-St2}, handles the case $rs=1$ and, therefore, also handles the case in which both $r$ and $a$ are powers of the same base and $b$ and $s$ are powers of the same base, so from here on we assume $rs>1$ and also assume that there do not exist nonnegative integers $h$, $k$, $w_1$, $w_2$, $z_1$, and $z_2$ such that $r = h^{w_1}$, $a=h^{w_2}$, $s=k^{z_1}$, and $b=k^{z_2}$.  

Now we can apply Lemma 6 with $(\alpha_1, \alpha_2, \alpha_3) = (r/s, a, b)$, $(A_1, A_2, A_3) =$ ($\log(\max(r,s))$, $\log(a)$, $\log(b)$), and $B = Z$ to get 
$$\log(1/\Lambda) < C A_1 A_2 A_3 \log(1.5 e B), \eqno{(11)}$$
 where $C = 1.6901816335 \cdot 10^{10}$.  Combining (10) and (11) we get 
$$
\log(d) < \log(c) + C  \log(\max(r,s)) \log(a) \log(b) \log(1.5 e Z). \eqno{(12)}$$
Also $\Lambda = \log(1+c/d) \le \max(\log(c), \log(2))$, so that, adding $\Lambda$ to both sides of (12), we get
$$
\log(D)< \max(\log(c), \log(2)) + \log(c) 
 + C  \log(\max(r,s)) \log(a) \log(b) \log(1.5 e Z).   
\eqno{(13)}$$
Using (13) and noting that $Z \log(j) \le \log(D) $, we have
$$ Z \log(j) < \max(\log(c), \log(2)) + \log(c)+ 
 C  \log(\max(r,s)) \log(a) \log(b) \log(1.5 e Z).$$
Dividing through by $\log(j)$ and noting $j \ge 2$, we get 
$$
Z < { \max(\log(c), \log(2)) + \log(c) \over \log(2)} 
+ C \log(\max(r,s)) \log(J) \log(1.5 e Z). 
\eqno{(14)}$$
Now applying Lemma 4, we have $c < (ZJ)^2$, so that (14) gives 
$$Z < { 4 \log(ZJ) \over \log(2)} + C \log(\max(r,s)) \log(J) \log(1.5 e Z ). $$
Now applying Lemma 5 we have 
$$Z < {8 \log(Z) \over \log(2)} + C (\log(Z))^2 \log(4.078 Z).$$
From this we obtain $Z < 8 \cdot 10^{14}$.  Applying Lemma 5 again proves the theorem.     
\end{proof} 

\begin{proof}[Proof of Corollary to Theorem 1:]  
The exceptions listed in the formulation of the corollary are all given by Theorem 1 of  \cite{Sc-St2}, which handles the case $rs=1$ and, therefore, handles the case in which $r$ and $a$ are both powers of the same base and $b$ and $s$ are both powers of the same base, so from here on we exclude these cases from consideration.  Since any solution to (1) can be rewritten as a case in which $a \notdiv r$, $b \notdiv s$, and $a$ and $b$ are not perfect powers, we use these restrictions in proceeding with a computer search.  

Suppose we have three solutions $(x_k, y_k, u_k, v_k)$ with $k=1$, 2, and 3.  For $i, j \in \{ 1, 2, 3\}$, we rewrite $(-1)^{u_i} r a^{x_i} + (-1)^{v_i} s b^{y_i} = c =(-1)^{u_j} r a^{x_j} + (-1)^{v_j} s b^{y_j}$ as 
$$r a^{x_0}(a^{x_h-x_0} +(-1)^m) = s b^{y_0}(b^{y_h - y_0} + (-1)^n ) \eqno{(15)}$$
where $x_0 = \min(x_i, x_j)$, $y_0 = \min(y_i, y_j)$, $x_h = \max(x_i, x_j)$, $y_h = \max(y_i, y_j)$,and $m$ and $n$ are in $\{0,1\}$.

For each choice of $(r,a,s,b)$ we use the technique known as \lq bootstrapping' (see \cite{GLS} and \cite{St1}) to find increasingly stringent congruence conditions on the exponents $x_h-x_0$ and $y_h-y_0$.  When these conditions show that either $x_h-x_0$ or $y_h-y_0$ exceeds $8 \cdot 10^{14}$, by Theorem 1 there can be no third solution.  (The bootstrapping in \cite{GLS} and \cite{St1} deals only with the case $m=n=1$ but the ideas extend easily to the other choices of signs, indeed, more easily since for $m=0$ or $n=0$ parity considerations often lead to a contradiction, as in the proof of Lemma 9 of \cite{Sc-St2}.  The Maple programs used can be found on the second author's website \cite{St2}.)  
Since we have three solutions to (1), Lemma 2 shows that for at least one pair of solutions we have $x_0 \ge 2$ (alternatively, $y_0 \ge 2$).  For a given $(r,a,s,b)$, we treat each pair $(m,n)$:  when $(m,n)=(0,0)$, $(0,1)$, or $(1,0)$, we take first $x_0 \ge 2$ and $y_0 \ge 1$, and then take $x_0 \ge 1$ and $y_0 \ge 2$, and use the bootstrapping method to find those few cases allowing two solutions; when $m=n=1$, we take $x_0 \ge 2$ and $y_0 \ge 2$ and bootstrap in the same way (to see that these lower bounds on $x_2$ and $y_2$ are justified, we note the following: if $(x_0, y_0)$ is a solution to (1) then Lemma 2 shows that we can get (15) with $\min(x_0, y_0) \ge 2$.  If $(x_0, y_0)$ is not a solution to (1), then it is not hard to see that we cannot have any two different solutions to (15) both having $m=n=1$ since then all three solutions to (1) would have the same $(u,v)$ which is impossible unless $(x_0, y_0)$ is a solution to (1); so we can assume that in this case at least one solution to (15) has $(m,n) \ne (1,1)$ with $\max(x_0, y_0) \ge 2$.)  
In the ranges under consideration, we find only a small number of choices of $(r,a,s,b)$ for which (1) could have three solutions.

Choose one such $(r,a,s,b)$; for each of the four possible choices of $(m,n)$, we use the bootstrapping method again to find numbers $k_x$ and $k_y$ such that if either $x_0>k_x$ or $y_0>k_y$ then (15) has no solution.  For each $(m,n)$, we take each pair $(x_0,y_0)$ such that $x_0 \le k_x$ and $y_0 \le k_y$; either bootstrapping finds a solution $(x_0, y_0, x_h, y_h, m,n)$ to (15) with the minimal $x_h$ and $y_h$, or bootstrapping shows no solution exists for that $(x_0, y_0)$.  In this way we obtain a complete list of quadruples $(x_0, y_0, m,n)$ that could occur in a solution of (15), for each of which we have found one solution $(x_0, y_0, x_h, y_h, m,n)$.  Suppose one such quadruple admits a second solution $(x_0, y_0, x_H, y_H, m,n)$.  Then we must have 
$$r a^{x_h} - s b^{y_h} = r a^{x_H} - s b^{y_H},$$
with $x_h<x_H$ and $y_h<y_H$, so that our list of quadruples must include $(x_h, y_h, 1,1)$.  We easily check that such $(x_h, y_h, 1,1)$ does not appear as one of the listed quadruples.  Thus, each quadruple on the list has only one solution, so that the associated solutions $(x_0, y_0, x_h, y_h, m,n)$ give all solutions to (15).  For each such solution, there are two possible values of $c$.  We compute all these $c$ values and find that no two are equal, confirming the corollary for this choice of $(r,a,s,b)$.  Continuing in this way for each $(r,a,s,b)$ completes the proof of the corollary.   
\end{proof}

The following theorem simply observes that cases of exactly two solutions to (1) are commonplace and easy to construct.  

\begin{Theorem} %2
There are an infinite number of cases of exactly two solutions to (1).  More specifically, for every choice of $(a,b)$, where $a$ and $b$ are relatively prime and not perfect powers, there are an infinite number of choices of $(x_1, y_1)$ such that for every quadruple $(a,b,x_1, y_1)$ there are an infinite number of quintuples $(r, s, c, x_2, y_2)$ with $(ra, sb) = (r,a) = (s,b) = 1$ such that both $(x_1, y_1)$ and $(x_2, y_2)$ give solutions to (1), and no further solution exists.  
\end{Theorem}

\begin{proof}  
We define $m(a,b)$ as the least value of $m > 1$ such that $b^n \pm 1 = a^m l$ for some integer $n$ and for some integer $l$ such that $\gcd(l,a)=1$.  
As pointed out in \cite{Be}, such $m(a,b)$ exists by, e.g., Ribenboim~\cite[C6.5]{Ri}.  (The definition here differs slightly from that in \cite{Be}.)  

For any choice of $(a,b)$ such that $(a,b)=1$ and $a$ and $b$ are not perfect powers, choose $x_1 \ge m(a,b)$ and $y_1 \ge m(b,a)$.  Then it is easily seen that we can find infinitely many choices of $r$, $s$, $x_2$, $y_2$, with $(ra, sb)= (r,a) = (s,b) = 1$, such that
$$r a^{x_1} (a^{x_2-x_1} \pm 1) = s b^{y_1} (b^{y_2-y_1} \pm 1). \eqno{(16)} $$
By Theorem 1, we can take $x_2$ sufficiently high to ensure no third solution exists.  
\end{proof}

\section{$(u, v) = (0, 1)$, with $\gcd(ra, sb)$ unrestricted and $\min(x,y) \ge 0$}  %3

Let $a>1$, $b>1$, $c>0$, $r>0$ and $s>0$ be positive integers.  
In this section we give a result on the number of solutions in nonnegative integers $(x,y)$ to the equation
$$r a^x - s b^y = c. \eqno{(17)}$$
Throughout this section we allow $(ra,sb)>1$ and $\min(x,y)=0$.  When $(ra,sb)=1$ and $\min(x,y)>0$, note that (17) is (1) with $(u,v) = (0,1)$.  

To eliminate cases of multiple solutions to (17) which are directly derived from other cases by multiplying each term of each solution by a fixed positive number $k$, we use the following definition: 
let $(x_1,y_1)$ be the least of $N$ solutions to (17); then, if there exists a positive integer $k>1$ such that $r a^{x_1} / k = r_1 a^w$ and $s b^{y_1} / k = s_1 b^z$ for positive integers $r_1$, $s_1$ and nonnegative integers $w$, $z$, the set of $N$ solutions is called {\it reducible}.  (Each irreducible set of solutions corresponds to an infinite number of reducible sets of solutions.)  

We also want to eliminate from consideration cases in which $a \mid r$ or $b \mid s$; we call such solutions {\it improper}.  

Finally, we want to eliminate from consideration cases where $a$ or $b$ is a perfect power; we call such solutions {\it redundant}.  

Just as it is easy to construct cases of exactly two solutions to (1) (see Theorem 2 above), it is easy to construct cases of exactly two solutions to (17), where the pair of solutions $(x_1, y_1)$ and $(x_2, y_2)$ is not reducible, improper, or redundant, since we can obtain $(ra, sb) = (r,a) = (s,b)=1$ as in (16).  (The solutions obtained in this way will have $\min(x,y)>0$.)  Note that in (16) we can always choose the signs on both the right and the left to be minus (when $a$ is even this may require taking $x_1 \ge m(a,b)+1$; similarly for $b$ even).

We now consider cases of more than two solutions to (17).  

\begin{Theorem}  %3
There are at most three solutions in nonnegative integers $(x,y)$ to (17).  There are an infinite number of cases of (17) with three solutions, even if we exclude reducible sets of solutions and also exclude solutions which are improper or redundant.     
\end{Theorem}

\begin{proof}  
Suppose (17) has more than two solutions and let the two least solutions be $(x_1, y_1)$ and $(x_2, y_2)$.  Let $(x_3, y_3)$ be a third solution, taking $x_1 < x_2 < x_3$ and $y_1 < y_2 < y_3$.  Let $R = {r a^{x_1} \over \gcd(r a^{x_1} , s b^{y_1})}$ and $S = {s b^{y_1} \over \gcd(r a^{x_1} , s b^{y_1})}$.   
We have 
$$R (a ^{x_2 - x_1} - 1) = S (b^{y_2 - y_1} - 1)$$
and
$$R (a ^{x_3 - x_1} - 1) = S (b^{y_3 - y_1} - 1).$$
Let
$$t = {a^{x_2-x_1} -1 \over S } = {b^{y_2 - y_1} - 1 \over R }$$
and
$$T = {a^{x_3-x_1} -1 \over S } = {b^{y_3 - y_1} - 1 \over R }.$$
Note that $t$ and $T$ are both integers.  

Let $g_1 = \gcd(x_2 - x_1, x_3-x_1)$ and $g_2 = \gcd(y_2-y_1, y_3-y_1)$.  
Let $k$ be the least integer such that $b^k - 1$ is divisible by $R$.  Then $k$ must divide both $y_2-y_1$ and $y_3-y_1$, so that $k$ divides $g_2$, and 
$$b^{g_2} - 1 = R l_2$$
for some integer $l_2$.  
Similarly, 
$$a^{g_1} - 1 = S l_1$$ 
for some integer $l_1$.  
Since $g_1$ divides both $x_2-x_1 $ and $x_3 -x_1$, $l_1$ divides $t$ and $T$.  
There must be an integer $j$ which is the least positive integer such that $b^j - 1$ is divisible by $R l_1$, and $j$ must divide both $y_2-y_1$ and $y_3-y_1$, so that $j$ divides $g_2$.  Therefore, $l_1 | l_2$.

A similar argument with the roles of $a$ and $b$ reversed shows that $l_2 | l_1$, so that $l_1 = l_2$, and we have 
$$r a^{x_1} (a^{g_1} - 1)= s b^{y_1} (b^{g_2} - 1). \eqno{(18)}$$
(18) shows that $(x_1+g_1, y_1+g_2)$ is a solution to (17).  If $x_1+g_1 \ne x_2$, then, by the definition of $x_2$, we must have $x_1+g_1>x_2$, which is impossible by the definition of $g_1$.   So $x_1+g_1=x_2$ and, similarly, $y_1+g_2=y_2$.  Letting $A = a^{x_2-x_1}$, $B = b^{y_2-y_1}$, $m={x_3-x_1 \over x_2-x_1}$ and $n={y_3-y_1 \over y_2-y_1}$, we find that we have a solution to the Goormaghtigh equation 
$${A^m - 1 \over A-1} = {B^n - 1 \over B-1} \eqno{(19)}$$
in integers $A>1$, $B>1$, $m>1$, $n>1$.  
Note that  both sides of (19) equal the integer $T/t$.  
By the same argument as above, if there exists a fourth solution $(x_4, y_4)$, we have a solution to (19) for the same values of $A$ and $B$ but with $m = {x_4-x_1 \over x_2 - x_1}$ and $n = {y_4-y_1 \over y_2-y_1}$, contradicting Theorem 1.3 of \cite{HT}, which states that, for given $A$ and $B$, (19) has at most one solution $(m,n)$.  

Finally, fixing $n=2$ and $m>2$, there are an infinite number of solutions $(A, B, m)$ to (19) each of which corresponds to a set of three solutions to (17) in which 
$$
(a, b, c, r, s, x_1, y_1, x_2, y_2, x_3, y_3) = (a_0, d A, A(d-1)/h, (dA-1)/h, (A-1)/h, 0,0, j,1, mj, 2) \eqno{(20)}  
$$
where $d = { A^{m-1} -1 \over A-1}$, $h = \gcd(dA-1, A-1)$, and $a_0^j=A$ with $a_0$ not a perfect power.  Suppose $b = dA$ is a perfect power.  Then there must be a prime $q \mid j$ such that $d = d_0^q$ and $A = a_1^q$ for some $d_0$ and $a_1 = a_0^{j/q}$.  This gives $d_0^q = { a_1^{q(m-1)} -1 \over a_1^q - 1}$, requiring $m>3$.  But this contradicts Corollary 1.2(b) of Bennett~\cite{Be2}, which states that the equation $(x^n - 1) / (x-1) = y^q$ in integers $x>1$, $y>1$, $n>2$, $q \ge 2$ has no solution $(x,y,n,q)$ where $x$ is a $q$-th power.  Thus, $b = dA$ is not a perfect power, so that the solutions given in (20) are not redundant.  Since $(r,a)=1$ and $b > s$, the solutions in (20) are not improper.  Finally, since $x_1=y_1 = 0$ and $(r,s)=1$, the set of solutions given by (20) is not reducible.   
\end{proof}

Remark:  Theorem 3 still holds if $\min(x,y)>0$, provided we revise the definition of reducibility to require $w>0$ and $z>0$.  For example, when $\min(x,y)>0$, (20) becomes 
$$(a, b, c, r, s, x_1, y_1, x_2, y_2, x_3, y_3) = (a_0, d A, d A^2(d-1)/h_1, d(dA-1)/h_1, (A-1)/h_1, j, 1, 2j, 2, (m+1)j, 3),$$
where $h_1 = \gcd(d(dA-1), A-1)$.  

The definitions of $A$, $B$, $m$, and $n$ in the proof of Theorem 3 show that any set of three solutions to (17) (for given $a$, $b$, $c$, $r$, $s$) corresponds to a unique solution $(A, B, m, n)$ to the Goormaghtigh equation (19).  On the other hand, each solution $(A,B,m,n)$ to (19) corresponds to an infinite number of sets of three solutions to (17).  However, it is not hard to show that each solution $(A,B,m,n)$ to (19) corresponds to a unique set of three solutions to (17) which is not reducible, improper, or redundant.  

The familiar Goormaghtigh conjecture says that, taking $A<B$, the only solutions $(A,B,m,n)$ to (19) with $n>2$ are given by $(A,B, m,n) = (2,5,5,3)$ and $(2,90,13,3)$.  (See \cite{Go} and \cite{Ra}.)  If this well-known conjecture is true, then the only two choices of $(a, b, c, r, s, x_1, y_1, x_2, y_2, x_3, y_3)$ which are not given by (20) and which give a set of three solutions to (17) which is not reducible, improper, or redundant are: 
$$(2,5,3,1,1,2,0,3,1,7,3), (2,90,88,89,1,0,0,1,1,13,3) \eqno{(21)}.$$ 
If $\min(x,y)>0$, (21) becomes $(2,5,15,5,1,2,1,3,2,7,4), (2, 90, 7920, 4005, 1, 1,1,2,2,14,4).$

\end{document}